\documentclass[leqno]{amsart}
\usepackage{amsthm}
\usepackage{color}
\usepackage{amssymb}
\usepackage{amsmath,amsfonts,enumerate}

\numberwithin{equation}{section}

\newtheorem{theorem}{Theorem}
\newtheorem{lemma}{Lemma}
\newtheorem{remark}{Remark}
\newtheorem{example}{Example}
\newtheorem{definition}{Definition}

\newtheorem{proposition}{Proposition}

\begin{document}

 \title{Correctness of the definition of the Laplace operator with delta-like potentials}
 
 \author{$^{1}$B. Kanguzhin}
 \email{kanguzhin53@gmail.com}
\address{Al-Farabi Kazakh National University, 050040, Almaty, Kazakhstan;
Institute of Mathematics and Mathematical Modeling, 050010, Almaty, Kazakhstan.}

\author{$^{2}$K. Tulenov}
\email{tulenov@math.kz}
\address{Al-Farabi Kazakh National University, 050040, Almaty, Kazakhstan;
Institute of Mathematics and Mathematical Modeling, 050010, Almaty, Kazakhstan.
}

\thanks{Corresponding authors: $^{1,2}$Al-Farabi Kazakh National University, 050040, Almaty, Kazakhstan;
$^{1,2}$Institute of Mathematics and Mathematical Modeling, 050010, Almaty, Kazakhstan. email:$^{1}$kanguzhin53@gmail.com and $^{2}$tulenov@math.kz }


\begin{abstract} In this paper, we give a correct definition of the Laplace operator with delta-like potentials.
Correctly solvable pointwise perturbation is investigated and formulas of resolvent are described. We study some properties of the resolvent. In particular, we prove Krein's formula for these resolvents.
\end{abstract}

 \maketitle

\section{Introduction}
 The main goal of this paper is to give a correct definition of formally defined operator via $-\Delta+\delta_{s},$ where $\Delta$ is Laplacian and
 $\delta_{s}$ is Dirac's delta function. A number of works have been devoted to this classical problem of Mathematical Physics. There are a lot of approaches in solving such a problem. One of the main approaches, based on the theory of Neumann on extensions of symmetric operators, is so called expansion to a wider space \cite{Pav}. Applications of Neuman's scheme of the theory of extensions, as a rule, are limited to a circle of abstract operators in a Hilbert space. In this note
we present a method of restrictions of previously well-defined maximal operator. The method of restrictions of maximal operator is the dual to the method of extensions of minimal operator. In the theory of extensions of the minimal operator, it is usually preferred to act in terms of the boundary form, i.e.
$$<B_{0}^{*}u,v>-<u,B_{0}^{*}v>,$$
where $B_{0}$ is the minimal symmetric operator. Further progress in various directions can be found in \cite{AK, Ben,BSSh, Frag, Kur, Shkal} from different points of view. A systematic application of non-standart analysis in the theory of point interactions can be found in \cite{AGHH}.

In this work the initial operator is the maximal operator $B_{M}$ which can be considered as the adjoint of the minimal operator $B_{0}$. After defining correctly the closed maximal operator $B_{M}$, its everywhere solvable restrictions will be constructed. We act in terms of the boundary form $<B_{M}u,v>-<u, B_{M}v>,$ when constructing the restrictions of the maximal operator. Hence, at the final stage, self-adjoint operators will be selected from everywhere solvable restrictions.

An alternative approach for the one-dimensional Sturm-Liouville operator with a potential $v(x),$ which is singular distribution of first order, was given in \cite{Shkal}. In particular, operators generated by a differential expression $-\frac{d^{2}}{dx^{2}}+v(x)$ were investigated, when $v(x)=\frac{da(x)}{dx}$ and
$a(x)\in L_{2}(0,1).$ Authors of \cite{Shkal} introduced definition of domain of the maximal operator as follows
$$D(B)=\{y(x),y^{[1]}(x)\in W_{2}^{1}[0,1]:-(y^{[1]})'-ay^{[1]}-a^{2}y\in
L_{2}(0,1)\},$$
where $y^{[1]}(x)=y'(x)-a(x)y(x)$ and $W_{2}^{1}[0,1]$ is Sobolev space. In case $v(x)=\delta
(x-x_{0})$, where $x_{0}\in(0,1),$ it is convenient to denote the domain $D(B)$ by $W_{2,\gamma}^{2}(\Omega_{0})$. Here
$\Omega_{0}=(0,1)\backslash \{x_{0}\},$
\begin{align*}&\gamma_{1}(y)=\lim_{\delta\rightarrow
0}(y(x_{0}+\delta)-y(x_{0}-\delta)),\\
&\gamma_{2}(y)=\lim_{\delta\rightarrow
0}(y'(x_{0}+\delta)-y'(x_{0}-\delta)).
\end{align*}
In this paper $D(B)$ is defined as follows
\begin{align*}&W_{2,\gamma}^{2}(\Omega_{0})=\{y\in
W_{2,loc}^{2}(\Omega_{0}):\gamma_{1}(y)=0,\gamma_{2}(y)=y(x_{0}),\\
&-(y^{[1]})'(x)-\frac{1}{2}sign(x)y^{[1]}(x)-\frac{1}{4}y(x)\in L_{2}(0,1)\}.
\end{align*}
Consequently, in \cite{Shkal}, authors constructed a well-defined restriction of the maximal operator $B$ in a standard way. Therefore, to construct maximal operator $B$ we need to introduce two specific functionals $\gamma_{1}(\cdot)$ and $\gamma_{2}(\cdot),$ and to define correctly domain of such operator via these functionals. If for $y\in W_{2,loc}^{2}(\Omega_{0})$ the values of functionals
$\gamma_{1}(y)$ and $\gamma_{2}(y)$ are finite, then the following equality holds
$$-(y^{[1]})'(x)-\frac{1}{2}sign(x)
y^{[1]}(x)-\frac{1}{4}y(x)=-\frac{d^{2}}{dx^{2}}(y(x)-\gamma_{2}(y)G(x,x_{0})),$$ whenever
 $\gamma_{1}(y)=0.$ Here $G(x,t)$ is the Green's function of the Dirichlet problem
\begin{align*} &u''(x)=f(x), \,\,\,\ 0<x<1\\
&u(0)=0, \,\,\ u(1)=0.
\end{align*}
in the space $L_{2}(0,1).$
If $y\in W_{2,\gamma}^{2}(\Omega_{0}),$ then
$$-\frac{d^{2}}{dx^{2}}(y(x)-\gamma_{2}(y)G(x,x_{0}))\in L_{2}(0,1). $$
In \cite{BSSh}, a correct definition of $-\Delta+\delta_{s}$ was given in $L_{2}(\mathbb{R}^{2})$ by a corresponding quadratic form defined in Sobolev space $H^{1}(\mathbb{R}^{2}).$ In \cite{Frag}, operators of the type $(-\Delta)^{m}+\delta_{s}$ defined correctly both on spaces $H^{2m}(\mathbb{R}^{n})$ and $H^{m}(\mathbb{R}^{n}).$ In addition, in the same paper, self-adjoint extension of the operator $(-\Delta)^{m}+\delta_{s}$ is given for wider classes than $H^{m}(\mathbb{R}^{n}).$ Calculations in \cite{Frag} are based on the results of \cite{BF}.
In this paper, we introduce a correct definition of the operator
$-\Delta+\delta_{s}$ in $L_{2}(\Omega),$ where $\Omega$ is a bounded set in $\mathbb{R}^{d}$
with smooth boundary. In order to reduce technical details, we set out our approach when $S$ is a singleton and $\Omega$ is a unit ball in $\mathbb{R}^{d}.$ Unlike \cite{BSSh}, in this paper the definition of the operator
$-\Delta+\delta_{s}$ is defined directly with no use of quadratic forms. Our definition has some own
advantages. In particular, it allows us to study properties
of resolvents of the operators in more detail. As a consequence, we may obtain some results on discrete spectrum of operators introduced in this work. The difference from \cite{Shkal} is that we investigate multi-dimensional case, i.e. $d>1.$ However, we consider potentials are represented by delta-function only, while in \cite{Shkal} authors considered potentials that can be singular distributions of the first order. Another difference of this paper from \cite{Shkal} is that the domain of the maximal operator contains functions which are not in $L_{2}(\Omega).$ Questions related to this paper are partially studied in \cite{KA,KTN,KT1,KT2} for the Laplace and polyharmonic operators. In particular, first regularised trace formulas for the Laplace operator were obtained in \cite{KT3}.

The structure of this paper is as follows.
Section 2 recalls the definition of maximal operator $B_{M}.$ Moreover, in Theorem \ref{B_{M} oper} we show the correct restriction of the maximal operator $B_{M}.$ It means problem \eqref{2.5}-\eqref{2.7} below has a unique solution in the punctured area $\Omega_{0}=\Omega\setminus S,$  where $S$ is singleton. Section 3 deals with correctly solvable pointwise perturbations. In section 4, formulas of resolvents of correctly solvable pointwise perturbations are given. As an application we obtain an analogue of M.G. Krein's formula for the resolvents. In section 5, we obtain some results when we change of initial part of spectrum of the Laplace operator.

\section{Definition of maximal operator}

In case $d=1$ and $S=\{x_{0}\},$ where $0<x_{0}<1,$ the definition of the domain of maximal operator defined in the previous section is given as follows:
\begin{align*} &D(B)=\{y\in W_{2,loc}^{2}(\Omega_{0}): \gamma_{1}(y)=0,
\gamma_{2}(y)=y(x_{0}),\\
&-\frac{d^{2}}{dx^{2}}(y(x)-\gamma_{2}(y)G(x,x_{0}))\in L_{2}(0,1)\},
\end{align*}
Moreover, we have
$By(x)=-\frac{d^{2}}{dx^{2}}(y(x)-\gamma_{2}(y)G_{1}(x,x_{0})),$
where $$G_{1}(x,x_{0})=\left\{
                                                                                              \begin{array}{ll}
                                                                                                x_{0}(1-x),\,\
                                                                                                $при$
                                                                                                \,\,\
                                                                                                x_{0}<x,
                                                                                                &
                                                                                                \\
                                                                                               x(1-x_{0}),
                                                                                               \,\
                                                                                               $при$
                                                                                               \,\,\
                                                                                               x\leq
                                                                                               x_{0}.&
                                                                                              \end{array}
                                                                                            \right.$$
Then a correct restriction of the operator $B$ is considered in terms of invertible operators, generated by operations
$-\frac{d^{2}}{dx^{2}}+\delta(x-x_{0}).$ In case $d>2$ and
$S=\{x_{0}\},$ where $x_{0}\in\Omega\subset\mathbb{R}^{d},$ first we need to define the domain of maximal operator, generated by operation $-\Delta+\delta_{s}(x),$ and then we will study its correct restriction.
For this purpose, we use the following well known facts.
\begin{theorem}\label{D prob} The solution of the Dirichlet problem for non-homogeneous harmonic equation
\begin{equation}\label{dirichlet}
\Delta w(x)=f(x), \,\,\ |x|<1
\end{equation}
with boundary condition
\begin{equation}\label{dirichlet cond}w\mid_{|x|=1}=0
\end{equation} is given by the following formula
\begin{equation}\label{dirichlet cond'}
w(x)=\int_{\Omega}G(x,\xi)f(\xi)d\xi,
\end{equation}
where
$G(x,\xi)=C\cdot(|x-\xi|^{2-d}-|\xi|^{2-d}\cdot|x-\frac{\xi}{|\xi|^{2}}|^{2-d})$ and $C$ is some constant.
\end{theorem}
{\bf Discussion of the Theorem \ref{D prob}.} Theorem \ref{D prob} claims that the Green function of the Dirichlet problem in the ball $\Omega,$ whenever $d>2,$ can be written out explicitly.
\begin{proof} It is well known from (\cite{Vlad}, p. 202) that the fundamental solution of harmonic equation in the unit ball
$\Omega=\{x\in\mathbb{R}^{d}:|x|<1\}$ is represented as
\begin{equation}\label{fund solution}\varepsilon(x,\xi)=C_{d}\cdot|x-\xi|^{2-d},
 \end{equation} where $C_{d}=-\frac{1}{(d-2)\sigma_{d}}$ and $\sigma_{d}$ is the area of $(d-1)$-dimensional surface of the unit sphere
$\partial\Omega$ in $\mathbb{R}^{d}.$ Set
\begin{align*} &X^{2}=|x-\xi|^{2}, \,\
Y^{2}=|\xi|^{2}\cdot|x-\frac{\xi}{|\xi|^{2}}|^{2},\\
&Z^{2}=(1-|x|^{2})\cdot(1-|\xi|^{2}).
\end{align*}
It is easy to prove $X^{2}\equiv
Y^{2}-Z^{2}.$ In particular, it follows that $Y^{2}\geq Z^{2}.$ Hence, the fundamental solution can be written as follows
\begin{align*} &\varepsilon(x,\xi)\stackrel{\eqref{fund solution}}=C_{d}\cdot
X^{2-d}=C_{d}\cdot(Y^{2}-Z^{2})^{\frac{2-d}{2}}=C_{d}\cdot\big(Y^{2}(1-\frac{Z^{2}}{Y^{2}})\big)^{\frac{2-d}{2}}\\
&=C_{d}\cdot Y^{2-d}(1-\frac{Z^{2}}{Y^{2}})^{1-\frac{d}{2}}=C_{d}\cdot
Y^{2-d}\big(1-\frac{2-d}{2}\cdot\frac{Z^{2}}{Y^{2}}+\frac{2-d}{2}\cdot\frac{1-d}{2}\cdot\frac{1}{2}\frac{Z^{4}}{Y^{4}}+...\big).
\end{align*}
Therefore, we obtain
\begin{equation}\label{green}C_{d}\cdot X^{2-d}-C_{d}\cdot Y^{2-d}=-C_{d}\cdot
 Y^{2-d}\big(\frac{2-d}{2}\cdot\frac{Z^{2}}{Y^{2}}-\frac{(2-d)(1-d)}{8}\cdot\frac{Z^{4}}{Y^{4}}+...\big)
\end{equation}
 We denote the left hand side of the identity \eqref{green} by $G(x,\xi),$ then
 $G(x,\xi)=C_{d}\cdot(X^{2-d}-Y^{2-d}).$ First, we show that the function $G(x,\xi)$
 is the Green function of Dirichlet problem for non-homogenous harmonic equation in a ball $\Omega.$ Since $\frac{\xi}{|\xi|^{2}}\in\Omega,$ it follows that $\Delta_{x}Y^{2-d}=0.$ Hence, by definition of the fundamental solution, we have $$\Delta_{x}G(x,\xi)=\Delta_{x}\big(\varepsilon(x,\xi)-C_{d}\cdot
 Y^{2-d}\big)=\delta(x,\xi).$$
Therefore,
 $G(x,\xi)$ is a solution of the differential equation $\Delta_{x}G(x,\xi)=\delta(x,\xi).$
 On the other hand, from \eqref{green} it follows that
 $$G(x,\xi)=-C_{d}\cdot
 Y^{2-d}\big(\frac{2-d}{2}\cdot\frac{Z^{2}}{Y^{2}}-\frac{(2-d)(1-d)}{8}\cdot\frac{Z^{4}}{Y^{4}}+...\big).$$
 Since $Z^{2}=(1-|x|^{2})(1-|\xi|^{2}),$ the trace on the boundary
 $Z^{2}\mid_{\partial\Omega}$ equals to zero. This implies the function  $G(x,\xi)$ on the boundary
 $\partial\Omega$ satisfies the Dirichlet boundary condition
 $$G(x,\xi)\mid_{x\in \partial\Omega}=0.$$ This completes the proof.
  \end{proof}

Let $\Omega$ be an open unit ball and let $x^{0}\in\Omega.$ We denote by $\Omega_{0}$ an open unit ball $\Omega$ without one fixed point $x^0,$ i.e. $\Omega_{0}:=\Omega\setminus \{x^{0}\}.$ For $\delta>0$ we denote the ball with radius $\delta$ and center at $x^{0}$ by $\Pi_{\delta}^{0}=\{x\in\Omega:|x-x^{0}|\leq \delta\}.$ For $s=0,1,2,...,d,$ define the following functionals
\begin{equation}\label{functional}\gamma_{0}(h)=-\lim_{\delta\rightarrow
0+}\int_{\partial\Pi_{\delta}^{0}}\frac{\partial}{\partial\nu_{t}}h(t)dS_{t}, \,\
\gamma_{s}(h)=d\cdot\lim_{\delta\rightarrow
0+}\int_{\partial\Pi_{\delta}^{0}}\frac{(t_{s}-x_{s}^{0})}{|t-x^{0}|}h(t)dS_{t}
\end{equation}
for some functions $h.$ Throughout this paper we denote
by $\frac{\partial}{\partial \nu_{\xi}}$ the normal derivative along the boundary $\partial\Pi_{\delta}^{0}$ at a point $\xi.$
We denote  by $W_{2,\gamma}^{2}(\Omega_{0})$ the elements $h$
in $W_{2,loc}^{2}(\Omega_{0})$ for which $\omega\mid_{\partial\Omega}=0$ with the values of
$\gamma_{0}(h),\gamma_{1}(h),...,\gamma_{d}(h)$  being finite, and satisfying
$$\Delta_{x}\big(h(x)-\gamma_{0}(h)\cdot G(x,x^{0})-\gamma_{1}(h)\cdot \frac{\partial
G(x,x^{0})}{\partial\xi_{1}}-...-\gamma_{d}(h)\cdot \frac{\partial
G(x,x^{0})}{\partial\xi_{d}}\big)\in L_{2}(\Omega).$$
The next Lemma shows that the space $ W_{2,\gamma}^{2}(\Omega_{0})$ can be larger than
$W_{2}^{2}(\Omega).$

\begin{lemma}\label{basic lem} For any $s=1,2,...,d,$ we have
$$G(x,x^{0})\in W_{2,\gamma}^{2}(\Omega_{0}), \,\, \frac{\partial
G(x,x^{0})}{\partial\xi_{s}}\in W_{2,\gamma}^{2}(\Omega_{0}).$$
Moreover,
$$\gamma_{0}(G)=-1, \, \gamma_{s}(G)=0, \, \gamma_{0}\big(\frac{\partial G}{\partial
\xi_{j}}\big)=0, \, \gamma_{s}\big(\frac{\partial G}{\partial \xi_{j}}\big)=\delta_{sj},$$
where $\delta_{sj}$ is the Kronecker symbol.
\end{lemma}
\begin{proof} Note that
the Green function $G(x,x^{0})$ is represented as a sum
$$G(x,x^{0})=\varepsilon(x,x^{0})+k(x,x^0),$$
where $\varepsilon(x,x^{0})$ is the fundamental solution (see \eqref{fund solution}) and $k(x,x^0)$ is compensating function.
Since compensating function $k(x,x^0)$ is smooth in $\Omega,$ it follows that
$$\gamma_{j}(k(\cdot,x^0))=0, \, \gamma_{j}(\frac{\partial}{\partial
t_{s}}k(\cdot,x^0))=0, \,\ j=0,1,...,d, \, s=1,2,...,d.$$
It remains to prove that the values of functionals
\begin{align*}&\gamma_{0}(\varepsilon),\gamma_{1}(\varepsilon),...,\gamma_{d}(\varepsilon),\\
&\gamma_{0}(\frac{\partial \varepsilon}{\partial \xi_{s}}),\gamma_{1}(\frac{\partial
\varepsilon}{\partial \xi_{s}}),...,\gamma_{d}(\frac{\partial \varepsilon}{\partial
\xi_{s}})
\end{align*}
are finite. By straightforward calculation, we obtain
\begin{align*}
 &\gamma_{0}(\varepsilon)\stackrel{\eqref{functional}}{=}-C_{d}\cdot\lim_{\delta\rightarrow
 0}\int_{\partial\Pi_{\delta}^{0}}\frac{\partial}{\partial
 \nu_{t}}|t-x^{0}|^{2-d}dS_{t}\\
&=-C_{d}\cdot\lim_{\delta\rightarrow
0}\int_{\partial\Pi_{\delta}^{0}}\sum_{k=1}^{d}\frac{\partial}{\partial
t_{k}}|t-x^{0}|^{2-d}\cdot\frac{(t_{k}-x^{0}_{k})}{|t-x^{0}|}dS_{t}\\
&=-C_{d}\sum_{k=1}^{d}\lim_{\delta\rightarrow
0}\int_{\partial\Pi_{\delta}^{0}}\frac{\partial}{\partial
t_{k}}(|t-x^{0}|^{2})^\frac{2-d}{2}\cdot\frac{(t_{k}-x^{0}_{k})}{|t-x^{0}|}dS_{t}\\
&=-C_{d}\sum_{k=1}^{d}\lim_{\delta\rightarrow
0}\int_{\partial\Pi_{\delta}^{0}}(2-d)|t-x^{0}|^{-d-1}(t_{k}-x^{0}_{k})^{2}dS_{t}\\
&=-C_{d}\cdot(2-d)\sum_{k=1}^{d}\lim_{\delta\rightarrow
0}\int_{\partial\Pi_{\delta}^{0}}|t-x^{0}|^{-d-1}(t_{k}-x^{0}_{k})^{2}dS_{t},
\end{align*}
where $C_{d}=-\frac{1}{(d-2)\sigma_{d}}.$
By using the substitution $t=x^{0}+\delta\cdot\frac{t-x^{0}}{\delta}=x^{0}+\delta\eta,$ here
$\eta\in S^{d-1}$ and $dS_{t}=\delta^{d-1}dS_{\eta},$ we obtain the required formula, i.e.
\begin{align*}
&\gamma_{0}(\varepsilon)=-C_{d}\cdot(2-d)\sum_{k=1}^{d}\lim_{\delta\rightarrow
0}\int_{S^{d-1}}\delta^{-d+1}\eta^{2}_{k}\delta^{d-1}dS_{\eta}\\
&=-C_{d}\cdot(2-d)\cdot\lim_{\delta\rightarrow
0}\int_{S^{d-1}}\sum_{k=1}^{d}\eta^{2}_{k}dS_{\eta}=-\frac{1}{(2-d)\sigma_{d}}(2-d)\sigma_{d}=-1.
\end{align*}
Similarly, again by the substitution $t=x^{0}+\delta\cdot\frac{t-x^{0}}{\delta}=x^{0}+\delta\eta,$
we compute $\gamma_{s}(\varepsilon)$
\begin{align*}
 &\gamma_{s}(\varepsilon)\stackrel{\eqref{functional}}{=}d\cdot C_{d}\lim_{\delta\rightarrow
 0}\int_{\partial\Pi_{\delta}^{0}}\frac{(t_{s}-x^{0}_{s})}{|t-x^{0}|}|t-x^{0}|^{2-d}dS_{t}\\
&=d\cdot C_{d}\cdot\lim_{\delta\rightarrow
0}\int_{S^{d-1}}\eta_{s}\delta^{2-d}\delta^{d-1}dS_{\eta}=d\cdot
C_{d}\cdot\lim_{\delta\rightarrow 0}\delta\cdot\int_{S^{d-1}}\eta_{s} dS_{\eta}=0.
\end{align*}
Now, we study the value of
$\gamma_{0}(\frac{\partial G(\xi,x^{0}}{\partial \xi_{j}})$. First, we need to compute the normal derivative of the function
$\frac{\partial \varepsilon(\xi,x^{0})}{\partial \xi_{j}}, \,\ j=1,2,...,d$
\begin{align*}
 &\frac{\partial}{\partial \nu_{t}}(\frac{\partial \varepsilon(t,x^{0})}{\partial
 \xi_{j}})=C_{d}\cdot(2-d)\sum_{k=1}^{d}\frac{\partial}{\partial
 t_{k}}\big(|t-x^{0}|^{-d}\cdot(t_{j}-x^{0}_{j})\big)\cdot\frac{(t_{k}-x^{0}_{k})}{|t-x^{0}|}\\
&=C_{d}\cdot(2-d)\big(\sum_{k=1}^{d}\big((t_{j}-x^{0}_{j})\cdot(-\frac{d}{2})|t-x^{0}|^{-d-2}\cdot2(t_{k}-x^{0}_{k})\cdot\frac{(t_{k}-x^{0}_{k})}{|t-x^{0}|}\big)+|t-x^{0}|^{-d-1}(t_{j}-x^{0}_{j})\big)\\
&=C_{d}\cdot(d-2)(d-1)(t_{j}-x^{0}_{j})\cdot|t-x^{0}|^{-d-1}.
\end{align*}
Next, we compute the values of the functional
$$\gamma_{0}(\frac{\partial \varepsilon(\cdot,x^{0})}{\partial
\xi_{j}})\stackrel{\eqref{functional}}{=}-C_{d}\cdot(d-2)(d-1)\cdot\lim_{\delta\rightarrow
0}\int_{\partial\Pi_{\delta}^{0}}(t_{j}-x^{0}_{j})\cdot|t-x^{0}|^{-d-1}dS_{t}.$$
By using the substitution $\eta=\frac{t-x^{0}}{\delta}$ in the last integral, we obtain the following equality
\begin{align*}
 \gamma_{0}(\frac{\partial \varepsilon(\cdot,x^{0})}{\partial \xi_{j}})
 &=-C_{d}\cdot(d-2)(d-1)\cdot\lim_{\delta\rightarrow
 0}\int_{\partial\Pi_{\delta}^{0}}(t_{j}-x^{0}_{j})\cdot|t-x^{0}|^{-d-1}dS_{t}\\
&=-C_{d}\cdot(d-2)(d-1)\lim_{\delta\rightarrow
0}\int_{S^{d-1}}\delta\cdot\eta_{j}\cdot\delta^{-d-1}\delta^{d-1} dS_{\eta}\\
&=-C_{d}\cdot(d-2)(d-1)\lim_{\delta\rightarrow 0}\big(\delta^{-1}\int_{S^{d-1}}\eta_{j}
dS_{\eta}\big)=0.
\end{align*}
Now, we compute the values of functionals $\gamma_{s}(\frac{\partial
\varepsilon(\cdot,x^{0})}{\partial \xi_{j}}),$ $s=1,2,...,d.$
If $j=s,$ then
\begin{align*}
 &\gamma_{s}(\frac{\partial \varepsilon(\cdot,x^{0})}{\partial
 \xi_{j}})\stackrel{\eqref{functional}}{=}C_{d}\cdot(2-d)\cdot d\cdot\lim_{\delta\rightarrow
 0}\int_{\partial\Pi_{\delta}^{0}}\frac{(t_{s}-x^{0}_{s})}{|t-x^{0}|}\cdot|t-x^{0}|^{-d}\cdot(t_{j}-x^{0}_{j})dS_{t}\\
 &=\frac{d}{\sigma_{d}}\cdot\lim_{\delta\rightarrow
 0}\int_{S^{d-1}}\eta_{s}\cdot\delta^{-d}\cdot\eta_{j}\delta^{d}dS_{\eta}=\frac{d}{\sigma_{d}}\cdot\lim_{\delta\rightarrow
 0}\int_{S^{d-1}}\eta_{s}\cdot\eta_{j}dS_{\eta}.
\end{align*}
Since
$$\int_{S^{d-1}}\eta^{2}_{j}dS_{\eta}=\frac{\sigma_{d}}{d},$$
it follows that
$$\gamma_{s}(\frac{\partial \varepsilon(\cdot,x^{0})}{\partial
 \xi_{j}})=1.$$

In the case $j\neq s,$ it is easy to see that
 $\gamma_{s}(\frac{\partial \varepsilon(\cdot,x^{0})}{\partial \xi_{j}})=0.$
This concludes the proof.
\end{proof}
\begin{lemma}\label{rep of h} Every element $h$ in $W_{2,\gamma}^{2}(\Omega_{0})$
is represented as
$$h(x)=h_{0}(x)+\gamma_{0}(h)G(x,x^{0})+\sum_{i=1}^{d}\gamma_{i}(h)\frac{\partial
G(x,x^{0})}{\partial
\xi_{i}}$$ where $h_{0}\in W_{2}^{2}(\Omega)$ and $h_{0}\mid_{\partial \Omega}=0.$

Moreover, such representation is unique.
\end{lemma}
\begin{proof} Since $h\in
W_{2,\gamma}^{2}(\Omega_{0}),$ it follows that there exist values of $\gamma_{0}(h),\gamma_{i}(h),i=1,2,...,d,$ which are finite. Set
$$w(x)\equiv
h(x)-\gamma_{0}(h)G(x,x^{0})-\sum_{i=1}^{d}\gamma_{i}(h)\frac{\partial
G(x,x^{0})}{\partial \xi_{i}}.$$ From the definition of
$W_{2,\gamma}^{2}(\Omega_{0})$ it follows that $w\mid_{\partial
\Omega}=0$ and $\Delta w\in L_{2}(\Omega).$

Let us denote $f(x)=\Delta w,\,x\in \Omega.$ Then, the solution of the Dirichlet problem
\begin{align*} &\Delta u=f(x),\,x\in \Omega,\\
&u\mid_{\partial \Omega}=0
\end{align*} exists and unique in the class
$W_{2}^{2}(\Omega).$ Consequently, $u(x)\in W_{2}^{2}(\Omega).$
The uniqueness of the solution of the Dirichlet problem implies that $u(x)\equiv
w(x).$ Therefore, $w(x)$ also belongs to $W_{2}^{2}(\Omega).$
This completes the proof.
\end{proof}

We correspond the maximal operator $B_{M}$ defined by the expression
$$B_{M}u=\Delta
\big(u-\gamma_{0}(u)G(x,x^{0})-\sum_{s=1}^{d}\gamma_{s}(u)\frac{\partial}{\partial
\xi_{s}}G(x,x^{0})\big)$$
\begin{align*}D(B_{M})=\big\{u\in W_{2,loc}^{2}(\Omega_{0}):
u\mid_{\partial\Omega}=0,u-\gamma_{0}(u)G(x,x^{0})-\sum_{s=1}^{d}\gamma_{s}(u)\frac{\partial}{\partial
\xi_{s}}G(x,x^{0})\in W_{2}^{2}(\Omega)\big\}.
\end{align*}
Therefore, domain of the maximal operator $B_{M}$ coincides with $W_{2,\gamma}^{2}(\Omega_{0}).$
\begin{remark} The maximal operator $B_{M}$ defined above is a closed operator in the following sense.
For all $n\geq 1,$ let us consider the following sequences
$$w_{n}(x)=w_{0 n}(x)+\sum_{i=0}^{d}\gamma_{i}(w_{n})\varphi_{i}(x)\in W_{2, \gamma}^{2}(\Omega_{0}),$$ where $w_{0 n}\in W_{2}^{2}(\Omega)$ and $w_{0 n}\mid_{\partial\Omega}=0,$ and
$$f_{n}(x)=-B_{M}w_{n}(x).$$
Suppose that $w_{0 n}(x)$ converges to $v_{0}(x)$ and $f_{n}(x)$ converges to $g(x)$ in $L_{2}(\Omega).$
Also, assume that there exist limits
$$\lim_{n\rightarrow \infty}\gamma_{i}(w_{n})=c_{i}, \,\ i=0,1,...,d.$$
Then $v_{0}\in W_{2}^{2}(\Omega),$ $\Delta v_{0}(x)=g(x),$ and
$$\lim_{n\rightarrow \infty}w_{n}(x)=v_{0}(x)+\sum_{i=0}^{d}c_{i}\varphi _{i}(x).$$
Moreover, it is easy to see that
$$B_{M}\left(v_{0}(x)+\sum_{i=0}^{d}c_{i}\varphi_{i}(x)\right)=-\Delta v_{0}(x).$$
This shows that the operator $B_{M}$ is closed in the above specified sense.
\end{remark}

The following theorem shows the correct restriction of the operator $B_{M}.$
\begin{theorem}\label{B_{M} oper} Boundary value problem for the Poisson equation in the punctured area $\Omega_{0}$
\begin{equation}\label{2.5}B_{M}u(x)=f(x), \,\,\ x\in\Omega_{0}
\end{equation}
with external Dirichlet condition
\begin{equation}\label{2.6}u\mid_{\partial\Omega}=0
\end{equation}
and with internal boundary condition
\begin{equation}\label{2.7}\gamma_{i}(u)=\gamma_{i}(h), \,\,\ i=0,1,2,...,d
\end{equation}
has a unique solution $u$ in the class $W_{2,\gamma}^{2}(\Omega_{0})$ for any $f\in L_{2}(\Omega)$ and $h\in
W_{2,\gamma}^{2}(\Omega_{0}).$
Moreover, it is represented by the formula
$$u(x)=\int_{\Omega}G(x,\xi)f(\xi)d\xi+\gamma_{0}(h)G(x,x^{0})+\sum_{i=1}^{d}\gamma_{i}(h)\frac{\partial
G(x,x^{0})}{\partial\xi_{i}}.$$
\end{theorem}
\begin{proof} To prove the Theorem \ref{B_{M} oper} we need to check the equation
\eqref{2.5}, external boundary condition \eqref{2.6}, and internal boundary condition \eqref{2.7}, which is easy to verify by Lemma \ref{basic lem}.
If $f=0$ and $h=0,$ then it follows from the theorem on removable singularity of harmonic function \cite[Theorem III. 39, p. 112]{O} that the problem \eqref{2.5}-\eqref{2.7} has a unique solution in the punctured area $\Omega_{0}.$
\end{proof}

\begin{remark} For the solvability of the problem \eqref{2.5}-\eqref{2.6}, it is necessary to add $d+1$ boundary conditions additionally. This fact was the motivation for the operator $B_{M}$ to be called maximal.
\end{remark}

Define a bilinear form $<f,w>$ of elements $f\in L_{2}(\Omega)$ and $w\in W_{2,\gamma}^{2}(\Omega_{0}).$ For this, first we need to find a function $v_{0}\in W_{2}^{2}(\Omega)$ as the solution of Dirichlet problem
\begin{align*} &\Delta v_{0}(x)=f(x),\,x\in \Omega,\\
&v_{0}\mid_{\partial \Omega}=0.
\end{align*}
Then, the bilinear form  $<f,w>$ is computed by
\begin{eqnarray*}\begin{split}<f,w>\equiv
&<-\Delta v_{0},w_{0}+\sum_{i=0}^{d}\gamma_{i}(w)\varphi_{i}(x)>:=<-\Delta v_{0}, w_{0}>_{L_2(\Omega)}+\sum_{i=0}^{d}\overline{\gamma_{i}(w)}\beta_{i}(v).
\end{split}\end{eqnarray*}
Here, we used the fact that by Lemma \ref{rep of h}, every $w\in W_{2,\gamma}^{2}(\Omega_{0})$ have a form
$$w(x)=w_{0}(x)+\sum_{i=0}^{d}\gamma_{i}(w)\varphi_{i}(x),$$
where $\varphi_{0}(x)=G(x,x^0),$ $\varphi_{i}(x)=\frac{\partial G(x,t)}{\partial t_{i}}\mid_{t=x^0},$ $i=1,...,d,$ and $w_{0}\in W_{2}^{2}(\Omega).$
Also, note that the numbers $\beta_{0}(v),\beta_{1}(v),...,\beta_{d}(v)$ are defined by formulas
$$\beta_{0}(v)=v_{0}(x^0), \,\ \beta_{i}(v)=\frac{\partial v_{0}(x)}{\partial x_{i}}\mid_{x=x^0}, \,\ i=1,...,d. $$
We further let
$$<w,f>=\overline{<f,w>}.$$
According to the Theorem \ref{B_{M} oper}, we compute the following boundary form
$$<B_{M}w,v>-<w,B_{M}v>.$$

\begin{theorem}\label{Boundary formula} For any $w,v\in W_{2,\gamma}^{2}(\Omega_{0}),$ we have
$$<B_{M}w,v>-<w,B_{M}v>=\sum_{i=0}^{d}\gamma_{i}(w)\overline{\beta_{i}(v)}-\sum_{i=0}^{d}\beta_{i}(w)\overline{\gamma_{i}(v)},$$
where $\beta_{0}(v)=v_{0}(x^0),$ $\beta_{i}(v)=\frac{\partial v_{0}}{\partial x_{i}}\mid_{x=x^0},$ $\beta_{0}(w)=w_{0}(x^0),$ $\beta_{i}(w)=\frac{\partial w_{0}}{\partial x_{i}}\mid_{x=x^0}.$
\end{theorem}
\begin{proof} By Lemma \ref{rep of h}, we know that every element $w$ in $W_{2,\gamma}^{2}(\Omega_{0})$
is represented uniquely as $$w(x)=w_{0}(x)+\gamma_{0}(w)G(x,x^{0})+\sum_{i=1}^{d}\gamma_{i}(w)\frac{\partial
G(x,x^{0})}{\partial \xi_{i}}$$ where $w_{0}\in W_{2}^{2}(\Omega)$ and $w_{0}\mid_{\partial \Omega}=0.$ Set
$\varphi_{0}(x)=G(x,x^0),$ $\varphi_{i}(x)=\frac{\partial G(x,x^0)}{\partial \xi_{i}},$ $i=1,2,...,d.$
Take two arbitrary elements $w,v\in W_{2,\gamma}^{2}(\Omega_{0}).$ We denote by $J(w,v)$ the boundary form $<B_{M}w,v>-<w,B_{M}v>.$
Then,
\begin{eqnarray*}\begin{split}J(w,v)
&=<-\Delta w_{0},v_{0}+\sum_{i=0}^{d}\gamma_{i}(v)\varphi_{i}>-<w_{0}+\sum_{i=0}^{d}\gamma_{i}(w)\varphi_{i},-\Delta v_{0}>\\
&=<-\Delta w_{0}, v_{0}>-<w_{0}, -\Delta v_{0}>-\sum_{i=0}^{d}\overline{\gamma_{i}(v)}<w_{0},\Delta \varphi_{i}>+\sum_{i=0}^{d}\gamma_{i}(w)<\varphi_{i},\Delta v_{i}>\\
&=\sum_{i=0}^{d}\gamma_{i}(w)\overline{\beta_{i}(v)}-\sum_{i=0}^{d}\beta_{i}(w)\overline{\gamma_{i}(v)}.
\end{split}\end{eqnarray*}
This completes the proof.
\end{proof}
It follows from the Theorem \ref{Boundary formula} that numerical vectors
$$\Gamma_{1}(w)=[\gamma_{0}(w),\gamma_{1}(w),...,\gamma_{d}(w)]$$
and
$$\Gamma_{2}(w)=[\beta_{0}(w),\beta_{1}(w),...,\beta_{d}(w)]$$ represent boundary trace of an element $w\in W_{2,\gamma}^{2}(\Omega_{0}).$ Since $w$ is arbitrary element from $W_{2,\gamma}^{2}(\Omega_{0}),$ the operators $\Gamma_{1}$ and $\Gamma_{2}$ are surjective from $W_{2,\gamma}^{2}(\Omega_{0})$ into $\mathbb{C}^{d+1}.$ Consequently, $<\mathbb{C}^{d+1},\Gamma_{1}, \Gamma_{2}>$ represents Boundary Triples (see \cite{AK, Ben, Kur} for more details). In \cite{Ben}, there were attended only boundary values $\gamma_{0}(w)$ and $\beta_{0}(w).$ In our case, there is more complete selection of boundary values $\Gamma_{1}(w)$ and $\Gamma_{2}(w).$

\section{Correctly solvable pointwise perturbations}

Define an operator $K$ which maps the elements from the space
$L_{2}(\Omega)$ into $W_{2,\gamma}^{2}(\Omega_{0}),$ which is continuous in the following sense:
\begin{enumerate}[{\rm (i)}]\label{A}
  \item If
the sequence of norms $\|h_{n}\|_{L_{2}(\Omega)}\rightarrow 0$
as $n\rightarrow \infty,$ then for any $j=0,1,2,...,d,$ we have
$\gamma_{j}(Kh_{n})\rightarrow 0$ as
$n\rightarrow \infty.$
\end{enumerate}
First, we prove the following important theorem.
\begin{theorem}\label{B_{M} oper2} For any operator
$$K:L_{2}(\Omega)\rightarrow W_{2,\gamma}^{2}(\Omega_{0}),$$
which is continuous in the sense (i), the following problem of pointwise perturbation
\begin{align*} &B_Mu(x)=f(x), \,\,\ x\in\Omega_{0}\\
&u\mid_{\partial\Omega}=0, \\
&\gamma_{j}(u)=\gamma_{j}(KB_Mu), \,\,\ j=0,1,2,...,d
\end{align*}
has a unique solution in $W_{2,\gamma}^{2}(\Omega_{0})$ for any $f\in L_2(\Omega),$ which is represented as
$$u(x)=\int_{\Omega}G(x,\xi)f(\xi)d\xi+\gamma_{0}(Kf)G(x,x^{0})+\sum_{i=1}^{d}\gamma_{i}(Kf)\frac{\partial
G(x,x^{0})}{\partial\xi_{i}}.$$
\end{theorem}
\begin{proof} Let $f\in L_2(\Omega).$ Suppose $h(x)=Kf(x).$ Then, applying the Theorem \ref{B_{M} oper} where we replace boundary conditions
$$\gamma_j(u)=\gamma_j(h), \,\ j=0,1,…,d$$ with conditions
$$\gamma_j (u-KB_Mu)= 0, \,\, j=0,1,…,d,$$ we can obtain new correctly solvable perturbations of the Dirichlet problem \eqref{dirichlet}-\eqref{dirichlet cond}.
Moreover, it holds for any $K$ in the Theorem \ref{B_{M} oper2}.
\end{proof}
{\bf Discussion of the Theorem \ref{B_{M} oper2}.} In the Theorem \ref{B_{M} oper2}, an operator, which is given by the expression
$\Delta-C_{0}(u)\delta(x-x^{0})-\sum_{s=1}^{d}C_{s}(u)\frac{\partial}{\partial
x_{s}}\delta(x,x^{0}),$ where functionals $C_{j}(u):=\gamma_{j}(KB_{M}u),$ $j=0,1,2,...,d,$ is well-defined. If
$C_{0}(u)=ku(x^{0}),$
$C_{1}=\cdots=C_{d}\equiv 0),$ then we obtain a correct definition of the operator $\Delta+k \delta(x-x^{0}).$ This fact was proved by A.M. Savchuk and A.A. Shkalikov in \cite{Shkal}, in case $d=1.$

Let us denote by $B_{K}$ an operator corresponding to the boundary value problem in the Theorem \ref{B_{M} oper2}. Then,
operator $B_{0}=\Delta$ corresponds to the Dirichlet problem \eqref{dirichlet}-\eqref{dirichlet cond}.
Note that if $\gamma_{s}(u)=0, \, s=0,1,2,...,d$ for $u\in
W_{2,\gamma}^{2}(\Omega_{0}),$ then $B_{M}u=B_{0}u.$ Similarly,
if $\gamma_{j}(u)=\gamma_{j}(KB_{M}u), \, j=0,1,2,...,d ,$ then
$B_Mu=B_{K}u.$

\begin{remark} It is possible to obtain inverse statement of the Theorem \ref{B_{M} oper2}, but here we do not pursue such a goal.
\end{remark}
\begin{remark}
It is easy to verify that functionals $\gamma_{j}(Kf),$ $j=0,1,2,...,d,$ are linear continuous on $L_2(\Omega).$ Therefore, by the Riesz representation theorem, we have
$$\gamma_{j}(Kf)=\int_{\Omega}f(t)\overline{c_{j}(t)}dt,$$
where $c_{j}\in L_2(\Omega),$ $j=0,1,2,...,d.$
Thus, the domain of $B_K$ is defined as
\begin{align*}
D(B_K)=\{u\in D(B_M): u\mid_{\partial \Omega}=0, \, \gamma_{j}(u)=\int_{\Omega}B_Mu\cdot
\overline{c_j(t)}dt, \, j=0,1,...,d\}.
 \end{align*}
where $\{c_{0}(\cdot),...,c_{d}(\cdot)\}$ is a system of functions from $L_2(\Omega).$
The operator, which is continuous in the sense $(i),$
$K:L_{2}(\Omega)\rightarrow W_{2,\gamma}^{2}(\Omega_{0})$ is represented as
$$Kf(x)=h_0(x)+\gamma_0(Kf)G(x,x^0)+\sum_{j=1}^{d}\gamma_j(Kf)\frac{\partial
G(x,x^0)}{\partial\xi_j}, \, h_0\in W_{2}^2(\Omega).$$
Since $\gamma_s(h_0)=0$ for $s=0,1,...,d,$ and by the Theorem \ref{B_{M} oper2} the values of functionals $\gamma_{s}(Kf)$ are involved in the boundary conditions, therefore, it is sufficient to consider operators $K$
as
\begin{equation}\label{2.8}Kf(x)=C_0(f)G(x,x^0)+\sum_{j=1}^{d}C_j(f)\frac{\partial
G(x,x^0)}{\partial \xi_{j}},
\end{equation}
where $C_j(\cdot),$ $j=0,1,2,...,d,$ are linear continuous functionals on $L_2(\Omega)$.
\end{remark}
We denote by $\mathcal{K}_{d+1}$ a set of finite rank operators as in \eqref{2.8} for any systems $$\{C_0(\cdot), C_1(\cdot),...,C_d(\cdot)\}$$
where $C_s(\cdot)$ are linear continuous functionals on $L_2(\Omega).$ It follows from the Theorem \ref{B_{M} oper2} that
 for different operators $K\in\mathcal{K}_{d+1}$ we correspond different operators $B_K.$
 Hence, a family $\{B_K\}$ of operators $B_K$ can be parameterized by a parameter of an operator $K\in\mathcal{K}_{d+1}.$
If  $c_j(\cdot)\in W_2^2(\Omega)$ and $c_j(x^0)=0, \,
 \frac{\partial c_j(x^0)}{\partial x_s}=0,$ $s=1,...,d,$ then the domain of the operator $B_K$ takes the form
\begin{eqnarray*}\begin{split}
&D(B_K)=\{u\in D(B_M): u\mid_{\partial \Omega}=0,\\
&\gamma_{j}(u)=\int_{\Omega}u(t)\cdot \overline{\Delta
c_j(t)}dt-\int_{\partial\Omega}\frac{\partial u}{\partial\nu_{t}}\cdot\overline{c_j(t)}dt\\
&+\gamma_{0}(u)\int_{\partial\Omega}\frac{\partial G}{\partial\nu_{t}}\cdot\overline{c_j(t)}dt+\sum_{s=1}^{d}\gamma_{s}(u)\int_{\partial\Omega}\frac{\partial^{2} G}{\partial\nu_{t}\partial\xi_{s}}\cdot\overline{c_j(t)}dt,
\, j=0,1,...,d\}.
\end{split}\end{eqnarray*}
Moreover, if $\Delta c_j(t)=0$ in $\Omega,$ then
\begin{eqnarray*}\begin{split}
&D(B_K)=\{u\in D(B_M): u\mid_{\partial \Omega}=0, \\
&\gamma_{j}(u)=-\int_{\partial\Omega}\frac{\partial
u}{\partial\nu_{t}}\cdot\overline{c_j(t)}dt\\
&+\gamma_{0}(u)\int_{\partial\Omega}\frac{\partial G}{\partial\nu_{t}}\cdot\overline{c_j(t)}dt+\sum_{s=1}^{d}\gamma_{s}(u)\int_{\partial\Omega}\frac{\partial^{2} G}{\partial\nu_{t}\partial\xi_{s}}\cdot\overline{c_j(t)}dt, \, j=0,1,...,d\}.
\end{split}\end{eqnarray*}
 If we suppose $c_0(t)=k\cdot G(x,x^0), \, c_j(t)\equiv 0, \, j=1,...,d,$ then the domain of the operator
 $B_K$ has the form
 \begin{eqnarray*}&D(B_K)=\{u\in D(B_M): u\mid_{\partial \Omega}=0,\\
 &\gamma_0(u)=k\cdot\lim_{x\rightarrow x_{0}}(u-\gamma_{0}(u)G(x,x^{0})), \,\ \gamma_{j}(u)=0, \, j=1,...,d\}.
 \end{eqnarray*}
Consequently, in this case the operator $B_K$ generated by the differential expression
$\Delta+k\delta(x-x^0)$ is defined correctly. Moreover, by the Theorem \ref{B_{M} oper2} such operator is invertible from $L_2(\Omega)$ into $D(B_K).$

\begin{definition}\label{selfadjoint}An invertible restriction $B_{K}$ of the maximal operator $B_{M}$ is called self-adjoint respect to the boundary form
\begin{equation}\label{bound form}<B_{M}w,v>-<w,B_{M}v>, \,\ \forall w,v\in D(B_{M}),
\end{equation}
if for all $w$ and $v$ in $D(B_{K})$ the following equality holds
$$<B_{K}w,v>=<w,B_{K}v>.$$
\end{definition}
We select self-adjoint operators respect to the previous boundary form \eqref{bound form} from the set of invertible restrictions $B_{K}$ which described in the Theorem \ref{B_{M} oper2}.
\begin{theorem}\label{B_alpha}Let $\alpha$ be a set consisting of complex numbers $\alpha_{0},\alpha_{1},...,\alpha_{d}.$ An operator $B_{\alpha}$ corresponding to the boundary value problem
\begin{align*} &B_{M}w(x)=f(x), \,\,\ x\in\Omega_{0},\\
&w\mid_{\partial\Omega}=0,\\
&\gamma_{i}(w)=\alpha_{i}\cdot\beta_{i}(w), \,\,\ j=0,1,2,...,d,
\end{align*}
is an invertible restriction of the maximal operator $B_{M}.$

Moreover, the operator $B_{\alpha}$ is self-adjoint respect to
the boundary form \eqref{bound form}.
\end{theorem}
\begin{proof}
Let $\alpha=(\alpha_{0},\alpha_{1},...,\alpha_{d})$ be a set of complex numbers. Define an operator $K$ in the following form
$$Kf(x)=\sum_{i=0}^{d}\alpha_{i}<f,\varphi_{i}>\varphi_{i}(x).$$
Then, correct restriction $B_{K}$ coincides with the operator $B_{\alpha}.$ Consequently, $B_{\alpha}\subset B_{M}$ and there exists $B_{\alpha}^{-1}.$
Now, we need to show that $B_{\alpha}$ is a self-adjoint operator in the sense of Definition \ref{selfadjoint}.
According to the Theorem \ref{Boundary formula}, we have
$$<B_{M}w,v>-<w,B_{M}v>=\sum_{i=0}^{d}\gamma_{i}(w)\overline{\beta_{i}(v)}-\sum_{i=0}^{d}\beta_{i}(w)\overline{\gamma_{i}(v)}.$$

If we put the operator $B_{K}$ instead of $B_{M}$ in the previous formula, then we obtain
$$<B_{\alpha}w,v>-<w,B_{\alpha}v>=0,$$
for all $w,v\in D(B_{\alpha}).$
This concludes the proof.
\end{proof}

\begin{remark} In \cite{Ben, Kur}, it was described self-adjoint extensions of minimal operator which can be both invertible and non-invertible operators.
In the Theorem \ref{B_alpha}, it is selected only invertible self-adjoint operators.
\end{remark}

\section{Resolvents of correctly solvable pointwise perturbations}

In this section, we write out directly a representation of resolvents of boundary value problems from the Theorem \ref{B_{M} oper2}.
\begin{theorem}\label{B_{M} oper3} Let us given a linear continuous operator
$$K:L_{2}(\Omega)\rightarrow W_{2,\gamma}^{2}(\Omega_{0}).$$
Then the following problem of pointwise perturbation
\begin{align*} &(B_{M}-\lambda)u(x)=f(x), \,\,\ x\in\Omega_{0},\\
&u\mid_{\partial\Omega}=0,\\
&\gamma_{j}(u)=\gamma_{j}(KB_{M}u), \,\,\ j=0,1,2,...,d
\end{align*}
has a unique solution in the space $W_{2,\gamma}^{2}(\Omega_{0})$ for any $f\in L_2(\Omega)$ and for any complex-valued spectral parameter $\lambda$, save possibly for some countable set.
Moreover, for resolvent $(B_{K}-\lambda I)^{-1},$ we have
\begin{equation}\begin{split}\label{2.9}(B_{K}-\lambda I)^{-1}f(x)
&=(B_{0}-\lambda I)^{-1}f(x)\\
&+\gamma_0 \big(KB_{0}(B_{0}-\lambda
I)^{-1}f\big)B_K(B_K-\lambda I)^{-1}G(x,x^{0})\\
&+\sum_{j=1}^{d}\gamma_j(KB_{0}(B_{0}-\lambda I)^{-1}f)B_{K}(B_K-\lambda
I)^{-1}\frac{\partial G(x,x^{0})}{\partial t_j}.
\end{split}\end{equation}
\end{theorem}
\begin{proof}
We have,
\begin{align*}B_{0}(B_{0}-\lambda I)^{-1}=I+\lambda(B_{0}-\lambda I)^{-1}.\end{align*}
Similarly, we obtain
\begin{align*}B_{K}(B_{K}-\lambda I)^{-1}=I+\lambda(B_{K}-\lambda I)^{-1}.\end{align*}
Set
\begin{align*} &u_{0}(x)=(B_{0}-\lambda I)^{-1}f(x)\\
&T_{0}(x)=(B_{K}-\lambda I)^{-1}G(x,x^{0})\\
&T_{j}(x)=(B_{K}-\lambda I)^{-1}\frac{\partial G(x,x^{0})}{\partial t_j}.
\end{align*}
We define a function $u$ by the formula
 \begin{equation}\begin{split}\label{2.11}
u(x)=u_{0}(x)+\gamma_{0}(K(f+\lambda u_{0}(x)))(G(x,x^0)+\lambda
T_{0}(x))\\
+\sum_{j=1}^{d}\gamma_{j}(K(f+\lambda u_{0}(x)))(\frac{\partial G(x,x^{0})}{\partial
t_j}+\lambda T_{j}(x)).
\end{split}\end{equation}
Since $u_{0}\in W_{2}^{2}(\Omega),$  $T_{j}\in D(B_{K}),$
$j=0,1,...,d,$ it follows that $B_{M}u_{0}=B_{0}u_{0},$ $B_{M}T_{j}=B_{K}T_{j},$ and hence, we have
\begin{align*}B_{M}u(x)-\lambda u(x)\\
&=B_0u_{0}(x)-\lambda u_{0}(x)+\gamma_{0}(K(f+\lambda u_{0}(x)))(B_{M}(G(x,x^{0}))\\
&-\lambda G(x,x^{0})+\lambda(B_{K}(T_{0})-\lambda T_{0}(x)))\\
&+\sum_{j=1}^{d}\gamma_{j}(K(f+\lambda u_{0}(x)))(B_{M}(\frac{\partial G(x,x^0)}{\partial
t_j})-\lambda \frac{\partial G(x,x^0)}{\partial t_j}\\
&+\lambda(B_{K}(T_j(x))-\lambda T_j(x))).
\end{align*}
From
\begin{align*} &B_{0}u_0(x)-\lambda u_0(x)=f(x),\\
&B_{M}(G(x,x^0))=0,\\
&B_{M}(\frac{\partial G(x,x^0)}{\partial t_j})=0,\\
&B_{K}(T_j(x))-\lambda T_j(x)= \frac{\partial G(x,x^0)}{\partial t_j},\\
&B_{K}(T_0(x))-\lambda T_0(x)=G(x,x^0)
\end{align*}
we obtain the required formula
$$B_Mu(x)-\lambda u(x)=f(x).$$
Now, we compute the trace of the function $u$ on the external boundary
$\partial\Omega$. Note that functions $u_0(x),$ $G(x,x^0)$
are solutions of Dirichlet problem and their traces on the external boundary
$\partial\Omega$ equal to zero. Then, the trace of function $\frac{\partial
G(x,x^0)}{\partial t_j}$ also vanishes on the external boundary
$\partial\Omega.$ Since $u$ is a linear combination of above functions, it follows that the its trace is equal to zero on $\partial\Omega.$ It remains to calculate the values of the functionals $\gamma_j(u-KB_{M}u),$
$j=0,1,…,d.$ It is clear that $\gamma_j(w)= \gamma_j(KB_{M}w),$ $j=0,1,…,d$ when
$w(x)=T_j(x),$ since $T_{j}\in
D(B_{K})$. Similarly, we have $\gamma_{j}(u_{0})=0,$
since $u_{0}\in W_{2}^{2}(\Omega).$ Applying functionals $\gamma_{s},$ $s=0,1,...,d,$ to the both sides of
\eqref{2.11} and by the preceding, we obtain
\begin{align*}
&\gamma_{s}(u)=\gamma_{0}(K(f+\lambda
u_{0}))\delta_{0s}+\delta_{js}\sum_{j=1}^{d}\gamma_{j}(K(f+\lambda u_{0}))\\
&+\lambda\gamma_{0}(K(f+\lambda
u_{0})\gamma_{s}(KB_{K}T_{0})+\lambda\sum_{j=1}^{d}\gamma_{j}(K(f+\lambda
u_{0}))\gamma_{s}(KB_{K}T_{j}).
\end{align*}
On the other hand, since $B_{M}G=0,$ $\gamma_{s}(KB_{M}G)=0,$
$\gamma_{s}(KB_{M}\frac{\partial G}{\partial t_{j}})=0,$ it follows that
\begin{align*}&\gamma_{s}(KB_{M}u)=\gamma_{s}(K(f+\lambda
u_{0}))+\lambda\gamma_{0}(K(f+\lambda u_{0}))\gamma_{s}(KB_{K}T_{0})\\
&+\lambda\sum_{j=1}^{d}\gamma_{j}(K(f+\lambda u_{0}))\gamma_{s}(KB_{K}T_{j}).
\end{align*}
Thereby completing the proof of the Theorem \ref{B_{M} oper3}.
\end{proof}
\begin{remark} Theorem \ref{B_{M} oper3} says that to calculate the resolvent $(B_K-\lambda I)^{-1}$ for any $f,$ we need to know the values of $(B_K-\lambda
I)^{-1}$ at fixed elements $G(x,x^0),$ $\frac{\partial
G(x,x^0)}{\partial t_{j}},$ $j=1,2,...,d.$
\end{remark}
\begin{remark} Formulas \eqref{2.9} represent a generalization of the second Hilbert identity for the resolvent in the case, when $D(B_K)\neq D(B_0).$ 
 Similar identities to
\eqref{2.9} were studied in \cite{BS}.
\end{remark}
In particular, Theorem \ref{B_{M} oper3} implies that the difference of resolvents $(B_{K}-\lambda
I)^{-1}-(B_{0}-\lambda I)^{-1}$ is a finite rank operator, therefore, there exists finite trace
\begin{equation}\begin{split}\label{2.12} Tr\big((B_{K}-\lambda I)^{-1}-(B_{0}-\lambda
I)^{-1}\big)
&=\sum_{i=0}^{d}\gamma_{i}\big(KB_{0}(B_{0}-\lambda I)^{-1}B_{K}(B_{K}-\lambda
I)^{-1}\varphi_{i}\big)
\end{split}\end{equation}
where
$$\varphi_{0}(x)=G(x,x^{0}), \, \varphi_{i}(x)=\frac{\partial G(x,x^{0})}{\partial t_{i}},
\, i=1,...,d.$$
We will define below meromorphic function $\Delta(\lambda)$ and
prove an analogue of M.G. Krein's formula \cite{BY,GK}.
Set
\begin{align*} &\beta_{00}=\gamma_0(KB_0(B_0-\lambda I)^{-1}\varphi_{0}), \,
\beta_{j0}=\gamma_j (KB_0(B_0-\lambda I)^{-1}\varphi_{0})\\
&\beta_{0s}=\gamma_0\left(KB_0(B_0-\lambda I)^{-1}\varphi_{s}\right), \,
\beta_{js}=\gamma_j\left(KB_0(B_0-\lambda
I)^{-1}\varphi_{s}\right).
\end{align*}
Then we define characteristic determinant by the following formula
$$\Delta(\lambda)=(-1)^{d+1}\left[%
\begin{array}{cccc}
   \lambda\beta_{00}-1 & \lambda\beta_{10} & \cdots & \lambda\beta_{d0} \\
   \lambda\beta_{01} & \lambda\beta_{11}-1 & \cdots & \lambda\beta_{d0} \\
  \vdots & \vdots & \ddots & \vdots  \\
  \lambda\beta_{0d} & \lambda\beta_{1d} & \cdots & \lambda\beta_{dd}-1\\
\end{array}%
\right]_{d\times d}
$$
The determinant $\Delta(\lambda)$ is a perturbation determinant (see \cite[Chapter IV, p.156]{GK}).
\begin{theorem}\label{Krein's formula} Let $B_{K}$ and $B_{0}$ be operators defined as in the Theorem \ref{B_{M} oper3}. Then, the following equality holds
\begin{equation}\label{2.13}Tr\big((B_{K}-\lambda I)^{-1}-(B_{0}-\lambda I)^{-1}\big)=-\frac{d}{d
\lambda}\ln\big(\Delta(\lambda)\big).
\end{equation}
\end{theorem}
\begin{proof}
Let us define functions
$$\psi_{s}(x)=B_{K}(B_{K}-\lambda I)^{-1}\varphi_{s}.$$
for any $s=0,1,...,d$
It is easy to see that for a fixed $s$ the function $\psi_{s}$ is a solution of the following problem
  $$B_{M}\psi_{s}(x)=\lambda\psi_{s}(x), \,
  \gamma_{j}(\psi_{s})-\gamma_{j}(KB_{M}\psi_{s})=\delta_{js}, \, j=0,1,...,d.$$
Let us compute $\psi_{s}(x).$ For this, we define vector-column
$$\overrightarrow{P}(x)=\left(B_{0}(B_{0}-\lambda
I)^{-1}\varphi_{0}(x),...,B_{0}(B_{0}-\lambda
I)^{-1}\varphi_{d}(x)\right)^{T}.$$ Fix $s$ and denote by $M_{s}(x)$ the value of the determinant, which is obtained from the determinant $\Delta(\lambda)$ by replacing $s'$th column with the column
$\overrightarrow{P}(x).$ It is easy to verify that $M_{s}(x)$ is a solution of the equality
$$B_{M}M_{s}(x)=\lambda M_{s}(x).$$
To compute the difference
$\gamma_{i}(M_{s})-\gamma_{j}(KB_{M}M_{s}),$  first we consider
\begin{align*}\\
&\gamma_{i}(B_{0}(B_{0}-\lambda
I)^{-1}\varphi_{s})-\gamma_{i}(KB_{M}B_{0}(B_{0}-\lambda
I)^{-1}\varphi_{s})\\
&=\gamma_{i}(\varphi_{s})+\lambda\gamma_{i}((B_{0}-\lambda
I)^{-1}\varphi_{s})\\
&-\gamma_{i}(KB_{M}\varphi_{s})-\lambda\gamma_{i}(KB_{M}(B_{0}-\lambda
I)^{-1}\varphi_{s})\\
&=\delta_{is}-\lambda\gamma_{i}(KB_{0}(B_{0}-\lambda
I)^{-1}\varphi_{s})=\delta_{is}-\lambda\beta_{is}.
\end{align*}

Obviously,
$\gamma_{i}(M_{s})-\gamma_{j}(KB_{M}M_{s})=-\Delta(\lambda)\delta_{is}$
Consequently, we have
\begin{equation}\label{2.14}
\psi_{s}(x)=-\frac{M_{s}(x)}{\Delta(\lambda)}.
\end{equation}
After some adjustment of \eqref{2.12}, we obtain

$$Tr((B_{K}-\lambda I)^{-1}-(B_{0}-\lambda
I)^{-1})=\sum_{i=0}^{d}\gamma_{i}(KB_{0}(B_{0}-\lambda I)^{-1}\psi_{i}).$$
It follows from \eqref{2.14} that
\begin{equation}\label{2.15}Tr((B_{K}-\lambda I)^{-1}-(B_{0}-\lambda
I)^{-1})=-\frac{1}{\Delta(\lambda)}\sum_{i=0}^{d}\gamma_{i}(KB_{0}(B_{0}-\lambda
I)^{-1}M_{i})
\end{equation}
Let us define a column vector
$$\overrightarrow{V}(x)=\left(\varphi_{0}(x),...,\varphi_{d}(x)\right)^{T}.$$
For fix $s$ again we denote by $N_{s}(x)$ the value of the determinant, which is obtained from the determinant $\Delta(\lambda)$ by replacing $s'$th column with the column $\overrightarrow{V}(x).$
Since $M_{s}=B_{0}(B_{0}-\lambda I)^{-1}N_{s},$ it follows from \eqref{2.15}
that
\begin{equation}\label{2.16}
Tr\big((B_{K}-\lambda I)^{-1}-(B_{0}-\lambda
I)^{-1}\big)=-\frac{1}{\Delta(\lambda)}\sum_{i=0}^{d}\gamma_{i}\big(KB^{2}_{0}(B_{0}-\lambda
I)^{-2}N_{i}\big)
\end{equation}
Finally, note that
\begin{equation}\label{2.17}\frac{d}{d
\lambda}\Delta(\lambda)=\sum_{i=0}^{d}\gamma_{i}(KB^{2}_{0}(B_{0}-\lambda
I)^{-2}N_{i}),
\end{equation}
since
 $$\frac{d}{d
\lambda}\big(\lambda B_{0}(B_{0}-\lambda I)^{-1}\big)=\big(B_{0}(B_{0}-\lambda
I)^{-1}\big)^{2}.$$ Then, the combination of \eqref{2.16} and \eqref{2.17} gives us the required formula \eqref{2.13} of
M.G. Krein. This completes the proof.
\end{proof}
Let $f\in L_{2}(\Omega),$ then it is convenient to define the following expressions
\begin{align*} &\beta_0(f)=\gamma_0(KB_0(B_0-\lambda I)^{-1}f), \,\
\beta_j(f)=\gamma_j(KB_0(B_0-\lambda I)^{-1}f),\\
&F(f)=\beta_0(f)G(x,x^0)+\sum_{j=1}^{d}\beta_j(f)\frac{\partial G(x,x^0)}{\partial
t_j},\\
&F_0=(B_0-\lambda
I)^{-1}G(x,x^0)+\beta_{00}G(x,x^0)+\sum_{j=1}^{d}\beta_{j0}\frac{\partial
G(x,x^0)}{\partial
t_j},\\
&F_s=(B_0-\lambda I)^{-1}\frac{\partial G(x,x^0)}{\partial
t_s}+\beta_{0s}G(x,x^0)+\sum_{j=1}^{d}\beta_{js}\frac{\partial
G(x,x^0)}{\partial t_j},\\
&Q(f,x,\lambda)=\left(%
\begin{array}{ccccc}
\beta_0(f) & \beta_1(f) & \cdots & \beta_d(f) & F(f) \\
  \lambda\beta_{00}-1 & \lambda\beta_{10} & \cdots & \lambda\beta_{d0} & F_{0} \\
  \vdots & \vdots & \ddots & \vdots & \vdots \\
  \lambda\beta_{0d} & \lambda\beta_{1d} & \cdots &\lambda\beta_{dd}-1 & F_{d} \\
\end{array}%
\right).
\end{align*}

The next theorem gives a representation of resolvent $(B_{K}-\lambda
I)^{-1}.$
\begin{theorem} \label{B_{K} oper} Let the assumptions of Theorem \ref{B_{M} oper3} hold. Then, the resolvent
$(B_K-\lambda I)^{-1}$ is a finite dimensional perturbation of the resolvent $(B_0-\lambda I)^{-1},$ which is represented by the following formula
$$(B_K-\lambda I)^{-1}f(x)=(B_{0}-\lambda
I)^{-1}f(x)+\frac{Q(f,x,\lambda)}{\Delta(\lambda)}.$$
\end{theorem}
\begin{proof}
In the Theorem \ref{B_{M} oper3}, it is obtained a representation of the resolvent
$(B_K-\lambda I)^{-1}.$ We substitute consecutively
$G(x,x^0),$ $\frac{\partial
G(x,x^0)}{\partial t_s},$ $s=1,…,d,$ instead of $f(x)$ in the above representation.  As a result, we obtain a system of matrix-vector equations $$D\overrightarrow{P}=0,$$
where
\begin{align*} &\overrightarrow{P}=((B_{K}-\lambda I)^{-1}\varphi_{0},…,(B_{K}-\lambda
I)^{-1}\varphi_{d},-1)^{T},\\
&D=\left(%
\begin{array}{ccccc}
  \beta_0(f) & \beta_1(f) & \cdots & \beta_d(f) & F(f) \\
  \lambda\beta_{00}-1 & \lambda\beta_{10} & \cdots & \lambda\beta_{d0} & F_{0} \\
  \vdots & \vdots & \ddots & \vdots & \vdots \\
  \lambda\beta_{0d} & \lambda\beta_{1d} & \cdots & \lambda\beta_{dd}-1 & F_{d} \\
\end{array}%
\right).
\end{align*}
Since the previous system of linear algebraic equations has a non-trivial solution, the determinant of the system equals to zero, i.e. $detD=0.$ This implies
$$(B_K-\lambda I)^{-1}f(x)=(B_{0}-\lambda
I)^{-1}f(x)+\frac{Q(f,x,\lambda)}{\Delta(\lambda)},$$
which represents a generalization of the well known Hilbert's identity for resolvent. This concludes the proof.
\end{proof}
\begin{remark} Theorem \ref{B_{K} oper} represents a generalization of the 2nd Hilbert's identity for resolvent, whenever operators $B_{K}$ and $B_{0}$ have different domains.
\end{remark}
\begin{remark} Since $\beta_{sk},$ $F_s$ depend on a spectral parameter $\lambda$ in a meromorphic way, it follows that
$\Delta(\lambda)$ is a meromorphic function of
$\lambda.$ Consequently, the resolvent $(B_K-\lambda I)^{-1}$ is also a meromorphic operator function and its number of poles at most countable. This fact agrees with the statement of the Theorem \ref{B_{M} oper3}.
\end{remark}
Formula \eqref{2.9} can be generalized in the following way. For $s=1,...,d,$
we denote by $B_{s}$ an operator corresponding to the boundary value problem
\begin{align*} &(B_{M}-\lambda I)u=f,\\
&u\mid_{\partial \Omega}=0,\\
&\gamma_{j}(u)=\gamma_{j}(KB_{M}u), \, j=0,1,...,s-1,\\
&\gamma_{k}(u)=0, \, k=s,s+1,...,d.
\end{align*}
Then, we can obtain the following formula for the resolvent of the operator $B_{s}.$
\begin{equation}\begin{split}\label{2.10}(B_{s}-\lambda I)^{-1}f=(B_{s-1}-\lambda I)^{-1}f
&+\gamma_{s-1}\big(KB_{s-1}(B_{s-1}-\lambda I)^{-1}f\big)\frac{\theta_{s-1}}{1-\lambda \gamma_{s-1}(K\theta_{s-1})}
\end{split}\end{equation}
where $\theta_{s}=B_{s}(B_{s}-\lambda I)^{-1}\varphi_{s},$ $\varphi_{0}(x)=G(x,x^0)$,
$\varphi_{s}(x)=\frac{\partial G(x,x^0)}{\partial \xi_s},$
$s=1,2,...,d.$ In case $s=d,$ this implies equality $B_{K}=B_{d}.$ It follows from
\eqref{2.10}, $s=1,...,d,$ that operators $B_{1},...B_{d-1}$ can be excluded and we obtain formula \eqref{2.9}.

\section{Change of initial part of spectrum of the Laplace operator}

In this section, we study some spectral properties of correctly-solvable pointwise perturbation. Resolvents of such operators described in Theorem \ref{B_{K} oper}. In Theorem \ref{B_{K} oper}, the perturbation determinant $\Delta(\lambda)$ has appeared. Perturbation determinants were studied in \cite[Chapter IV, p. 156]{GK}. We consider operator $B_{K}$ as a perturbation of the operator $B_{0}$ and the operator $B_{0}=\triangle$ corresponds to the Dirichlet problem \eqref{dirichlet}-\eqref{dirichlet cond}. It is well known that $B_{0}$ is a self-adjoint operator with discrete spectrum. We denote the eigenvalues of the operator $B_0$ and the corresponding eigenvectors by $\{\mu_n\}_{n\geq 1}$ and $\{\omega_n(x)\}_{n\geq 1},$ respectively. It is known that the system of eigenvectors $\{\omega_n(x)\}_{n\geq 1}$ forms orthonormal basis in the space $L_2(\Omega).$ Then, the elements $\beta_{ij}$ of perturbation determinant $\Delta(\lambda)$ are identified by formula
 \begin{equation}\label{determinan}
\beta_{ij}=\gamma_{i}\big(KB_{0}(B_{0}-\lambda I)^{-1}\varphi_{j}\big)=\sum_{n=1}^{\infty}\frac{\mu_{n}<\varphi_{j},\omega_{n}>}{\mu_{n}-\lambda}\cdot C_{in},
\end{equation}
where
$C_{in}=\gamma_{i}(K\omega_{n}).$
\begin{proposition}\label{L3} For all $n\in \mathbb{N},$  we have
$$<\varphi_{0},\omega_{n}>=\frac{2\omega_{n}(x^{0})}{\mu_{n}}.$$
\end{proposition}
\begin{proof}
Since
$$\lim_{\delta\rightarrow +0}\int_{\partial
\Pi_{\delta}^{0}}\frac{\partial \varphi_{0}}{\partial
\nu_{\xi}}dS_{\xi}=1$$
and
 $$\lim_{\delta\rightarrow +0}\int_{\partial
\Pi_{\delta}^{0}}\frac{(\xi_{j}-x_{j}^{0})}{|\xi-x^{0}|}\varphi_{0}dS_{\xi}=0,$$
it follows that
\begin{align*}
&<\varphi_{0},\omega_{n}>=\lim_{\delta\rightarrow
+0}\int_{\Omega/\Pi_{\delta}^{0}}\varphi_{0}\omega_{n}(x)dx=\lim_{\delta\rightarrow+0}\frac{1}{\mu_{n}}\int_{\Omega/\Pi_{\delta}^{0}}\varphi_{0}\Delta\omega_{n}(x)dx\\
&=\frac{\omega_{n}(x^{0})}{\mu_{n}}-\lim_{\delta\rightarrow
+0}\frac{1}{\mu_{n}}\int_{\partial \Pi_{\delta}^{0}}\varphi_{0}\frac{\partial
\omega_{n}}{\partial \nu_{\xi}}dS_{\xi}+\lim_{\delta\rightarrow +0}\frac{1}{\mu_{n}}\int_{\partial
\Pi_{\delta}^{0}}\omega_{n}(\xi)\frac{\partial
\varphi_{0}}{\partial \nu_{\xi}}dS_{\xi}\\
&=\frac{\omega_{n}(x^{0})}{\mu_{n}}-\sum_{j=1}^{d}\frac{\partial
\omega_{n}(x^{0})}{\partial \xi_{j}}\frac{1}{\mu_{n}}\lim_{\delta\rightarrow
+0}\int_{\partial\Pi_{\delta}^{0}}\frac{(\xi_{j}-x_{j}^{0})}{|\xi-x^{0}|}\varphi_{0}dS_{\xi}\\
&+\frac{\omega_{n}(x^{0})}{\mu_{n}}\lim_{\delta\rightarrow
+0}\int_{\partial \Pi_{\delta}^{0}}\frac{\partial
\varphi_{0}}{\partial
\nu_{\xi}}dS_{\xi}=\frac{2\omega_{n}(x^{0})}{\mu_{n}},
\end{align*}
This concludes the proof.
\end{proof}
For all $s=1,2,...,d$ we define operators whose resolvents are computed recurrently by
\begin{align*}(B_s-\lambda I)^{-1}f=(B_{s-1}-\lambda
I)^{-1}f+\gamma_{s-1}(KB_{s-1}(B_{s-1}-\lambda
I)^{-1}f)\cdot\frac{\theta_{s-1}}{1-\lambda\gamma_{s-1}(K\theta_{s-1})},\end{align*}
where $\theta_{i}=B_i(B_{i}-\lambda I)^{-1}\varphi_{i},$ $i=0,1,...,d.$
For $s=1,$ define perturbation determinant
\begin{align*}
&\Delta_{01}(\lambda)=1-\lambda\gamma_0(K\theta_0)=1-\lambda\gamma_0\big(KB_0(B_0-\lambda I)^{-1}\varphi_0\big)
\end{align*}
Since
$$\Delta_{01}(\lambda)=1-\lambda\sum_{n=1}^{\infty}\frac{\mu_{n}<\varphi_{0},\omega_{n}>}{\mu_{n}-\lambda}\cdot C_{0n},$$
it follows from Proposition \ref{L3} that
\begin{equation}\label{delta01}
\Delta_{01}(\lambda)=1-2\lambda\sum_{n=1}^{\infty}\frac{\omega_{n}(x^{0})}{\mu_{n}-\lambda}\cdot C_{0n}.
\end{equation}
Formula \eqref{delta01} implies following Theorem.
\begin{theorem}\label{L4} If one of the conditions $\omega_{N}(x^{0})=0$ or $\gamma_{0}(K\omega_{N})=0$ holds for some $N,$ then we have
$$\lambda_{N}(B_{1})=\mu_{N},$$
where $\lambda_{N}(B_{1})$ is the $N'$th eigenvalue of the operator $B_{1}.$
\end{theorem}
\begin{proof} First, we prove the case, when the eigenvalues $\mu_{N}$ of the operator are simple. The proof below will require slight modification for the case of multiple eigenvalues $ \mu_{N}.$  Since
$$(B_1-\lambda I)^{-1}f(x)=(B_{0}-\lambda
I)^{-1}f(x)+\gamma_{0}\big(KB_{0}(B_{0}-\lambda I)^{-1}f\big)\cdot\frac{\theta_{0}(x)}{\Delta_{01}(\lambda)},$$
it follows
from $\omega_{N}(x^{0})=0$ that the residue is calculated by formula
$$res_{\mu_{N}}(B_{1}-\lambda I)^{-1}f=-<f,\omega_{N}>\omega_{N}(x)+\frac{\theta_{0}(x,\lambda)}{\Delta_{01}(\lambda)}\mid_{\lambda=\mu_{N}}\cdot\mu_{N}<f,\omega_{N}>\cdot C_{0N}.$$
This implies that $\lambda=\mu_{N}$ is an eigenvalue of the operator $B_{1}$ corresponding to the following eigenvector
$$\omega_{N}(x)+\frac{\mu_{N}C_{0N}}{\Delta_{01}(\mu_{N})}\cdot\theta_{0}(x,\mu_{N}).$$
If $\gamma_{0}(K\omega_{N})=0,$ the residue is calculated by formula
\begin{eqnarray*}\begin{split} &res_{\mu_{N}}(B_{1}-\lambda I)^{-1}f=-<f,\omega_{N}>\omega_{N}(x)\\
&+\frac{\gamma_{0}\big(KB_{0}(B_{0}-\lambda I)^{-1}f\big)}{\Delta_{01}(\lambda)}\mid_{\lambda=\mu_{N}}\omega_{N}(x^{0})\omega_{N}(x).
\end{split}\end{eqnarray*}
It follows from the last formula that $\lambda=\mu_{N}$ is again an eigenvalue of the operator $B_{1}$ corresponding to the  eigenvector $\omega_{N}(x).$
This completes the proof.
\end{proof}
\begin{remark} Theorem \ref{L4} can be easily reformulated for the pair of $B_{s-1}$ and $B_{s}.$
\end{remark}
\begin{remark} It can be easily seen from the Theorem \ref{L4} that the operator $B_{1}$ is not a self-adjoint operator as $B_{0}.$ \end{remark}
Next we give an example of a boundary problem in a punctured domain, when the spectrum of the Dirichlet problem remains unchanged.
\begin{example}\label{Ex1} If the operator $K$ maps all eigenvectors $\omega_{n}(x)$ of the operator $B_{0}$ to the linear combination of derivatives $\frac{\partial}{\partial \xi_{j}}G(x,x^{0}),$  $j=1,2,...,d,$ then the assumptions of Theorem \ref{L4} hold. In other words, we have $C_{0n}=0$ for all $n.$ Consequently, the spectrum of the operator $B_{1}$ coincides with the spectrum of $B_{0}.$
\end{example}
This example can be easily modified so that only a finite number of eigenvalues of the operator $B_{1}$ differs from those of $B_{0}.$ Similar examples for the Sturm-Liouville operators are usually studied by the method of Crum \cite{C}, whenever the only finite part of the spectrum is changed.

Note that the regularized trace of perturbed operator $B_{1}$ in Example \ref{Ex1} equals to zero. Regularized traces have been studied by many authors, in particular \cite{GK63, GK64}. For the perturbation in the Example \ref{Ex1}, the conditions of results of \cite[Theorem 1]{GK63} (see also \cite{GK64}), in which their results are valid, are violated.

\section{Acknowledgment}

The work was partially supported by the grant No. AP05131292 of the Science Committee of the Ministry of Education and Science of the Republic of Kazakhstan.

\end{document}